\documentclass[11pt, leqno]{amsart}

\usepackage[svgnames, dvipsnames, table]{xcolor}
\usepackage[colorlinks = true, allcolors = UCRceleste, pagebackref=true, colorlinks]{hyperref}

\usepackage[utf8]{inputenc}
\usepackage{caption}
\usepackage{subcaption}
\usepackage{float}
\usepackage[ruled, linesnumbered, algosection]{algorithm2e}
\usepackage{mathrsfs}
\usepackage{verbatim}
\usepackage{xcolor}
\usepackage{color}
\usepackage{graphicx}
\usepackage{rotating}
\usepackage{diagbox}
\usepackage{amssymb}
\usepackage{epstopdf}
\usepackage{tikz}
\definecolor{UCRceleste}{RGB}{0,192,243}
\definecolor{mintgreen}{RGB}{152,255,152}
\definecolor{pinksalmon}{RGB}{255,102,102}
\definecolor{hueso}{RGB}{245,245,220}
\definecolor{marfil}{RGB}{255,253,208}
\definecolor{amarillo}{RGB}{255,255,0}
\usetikzlibrary{decorations.markings,arrows}
\usetikzlibrary{decorations.pathreplacing}
\usepackage[inner=1.0in,outer=1.0in,bottom=1.0in, top=1.0in]{geometry}

\usepackage{enumitem}

\usepackage{orcidlink}

\numberwithin{equation}{section}

\newtheorem{theorem}{Theorem}[section]
\newtheorem{lemma}[theorem]{Lemma}
\newtheorem{proposition}[theorem]{Proposition}

\theoremstyle{definition}

\theoremstyle{remark}
\newtheorem{remark}[theorem]{Remark}

\makeatletter
\def\moverlay{\mathpalette\mov@rlay}
\def\mov@rlay#1#2{\leavevmode\vtop{%
   \baselineskip\z@skip \lineskiplimit-\maxdimen
   \ialign{\hfil$\m@th#1##$\hfil\cr#2\crcr}}}
\newcommand{\charfusion}[3][\mathord]{
    #1{\ifx#1\mathop\vphantom{#2}\fi
        \mathpalette\mov@rlay{#2\cr#3}
      }
    \ifx#1\mathop\expandafter\displaylimits\fi}
\makeatother

\newcommand{\suchthat}{\;\ifnum\currentgrouptype=16 \middle\fi|\;}

\newcommand{\Z}{\mathbb{Z}}

\newcommand{\Q}{\mathbb{Q}}
\newcommand{\R}{\mathbb{R}}

\newcommand{\epsq}{\epsilon_q}
\newcommand{\epsk}{\epsilon_k}
\newcommand{\ov}{\overline{p}}
\newcommand{\A}{\widetilde{A}}
\newcommand{\ct}{\widetilde{C}}
\newcommand{\Rt}{\widetilde{R}}

\newcommand{\legendre}[2]{\ensuremath{\left( \frac{#1}{#2} \right) }}
\newcommand{\dlegendre}[2]{\ensuremath{\left( \dfrac{#1}{#2} \right) }}
\newcommand{\htk}{\widehat{t_k}}
\def\abs#1{\left\vert{#1}\right\vert}

\def\O_K{{\Cal{O}_{K}}}
\def\O_F{{\Cal{O}_{F}}}
\def\N_F{{\Cal{N}_{F/\Q}}}

\begin{document}

\title{Efficient computation of the overpartition function and applications
}

\author[A. Barquero-Sanchez et al.]{Adrian Barquero-Sanchez\orcidlink{0000-0001-7847-2938},  Gabriel Collado-Valverde,  Nathan C. Ryan\orcidlink{0000-0003-4947-586X},  Eduardo Salas-Jimenez\orcidlink{0009-0003-0664-611X},  Nicolás Sirolli\orcidlink{0000-0002-0603-4784} and Jean Carlos Villegas-Morales\orcidlink{0009-0004-8933-0627}}

\address{Escuela de Matem\'atica, Universidad de Costa Rica, San Jos\'e 11501, Costa Rica}
\email{adrian.barquero\_s@ucr.ac.cr}

\address{Escuela de Matem\'atica, Universidad de Costa Rica, San Jos\'e 11501, Costa Rica}
\email{gabriel.collado@ucr.ac.cr}

\address{Department of Mathematics, Bucknell University, Lewisburg, Pennsylvania 17837} 
\email{nathan.ryan@bucknell.edu}

\address{Escuela de Matem\'atica, Universidad de Costa Rica, San Jos\'e 11501, Costa Rica}
\email{eduardo.salas@ucr.ac.cr}

\address{Departamento de Matemática - FCEyN - UBA and IMAS - CONICET, Pabellón I, Ciudad Universitaria, Ciudad Autónoma de Buenos Aires (1428), Argentina}
\email{nsirolli@dm.uba.ar}

\address{Escuela de Matem\'atica, Universidad de Costa Rica, San Jos\'e 11501, Costa Rica}
\email{jean.villegas@ucr.ac.cr}

\begin{abstract}
In this paper we develop a method to calculate the overpartition function
efficiently using a Hardy-Rademacher-Ramanujan type formula, and we use this
method to find many new Ramanujan-style congruences, whose existence is predicted by Treneer and a few of which were first discovered by Ryan, Scherr, Sirolli and Treneer.
\end{abstract}

\maketitle

\section{Introduction}

The representation of integers as sums of special families of integers has been a central topic in number theory since antiquity. Among these, the ways in which a positive integer $n$ can be expressed as a sum of positive integers has been the object of intense study since the early 1900's. Each such representation is called a \textit{partition} of $n$, and the number of different partitions of $n$ (where the order of the summands is not taken into account) is given by the partition function $p(n)$. For example, since the partitions of $4$ are
\begin{align*}
&4\\
&3 + 1\\
&2 + 2\\
&2 + 1 + 1\\
&1 + 1 + 1 + 1
\end{align*}
we have $p(4) = 5$.

It is well known that the values of the partition function grow very fast. For example $p(200) = 3972999029388$. In fact, in a landmark paper from 1918, Hardy and Ramanujan \cite{HR18} proved that $p(n)$ satisfies the asymptotic formula
\begin{align*}
p(n) \sim \frac{1}{4n\sqrt{3}} e^{\pi \sqrt{2n/3}} \quad \text{as $n \to \infty$.}
\end{align*}

Another milestone for the subject, also from Ramanujan (\cite{Ramanujan21}), is the statement of the famous congruences
\begin{align*}
p(5n+4)&\equiv 0\pmod{5},\\
p(7n+5)&\equiv 0\pmod{7}, \;\mathrm{and}\\
p(11n+6)&\equiv 0\pmod{11}.
\end{align*}
This motivated the challenge of finding (families of) congruences for the partition function in arithmetic progressions.
We refer to \cite{ahlgrenono01} for details on this subject.


In this paper we study a variant of the integer partitions which are called \textit{overpartitions}. These were introduced by Corteel and Lovejoy in \cite{CL04}. More precisely, if $n$ is a positive integer, then an \textit{overpartition} of $n$ is a representation of $n$ as a sum of positive integers but with the added possibility that the first occurrence of each number in the representation may be overlined. For example, the overpartitions of $4$ are

\begin{center}
\begin{minipage}{0.2\textwidth}
\begin{flushright}
\begin{align*}
&4\\
&3 + 1\\
&2 + 2\\
&2 + 1 + 1\\
&1 + 1 + 1 + 1
\end{align*}
\end{flushright}
\end{minipage}
\begin{minipage}{0.2\textwidth}
\begin{flushleft}
\begin{align*}
&\overline{4}\\
&\overline{3} + 1\\
&\overline{2} + 2\\
&\overline{2} + 1 + 1\\
&\overline{1} + 1 + 1 + 1
\end{align*}
\end{flushleft}
\end{minipage}
\begin{minipage}{0.2\textwidth}
\begin{flushleft}
\begin{align*}
&\\
&\overline{3} + \overline{1}\\
&\\
&\overline{2} + \overline{1} + 1\\
&
\end{align*}
\end{flushleft}
\end{minipage}
\begin{minipage}{0.2\textwidth}
\begin{flushleft}
\begin{align*}
&\\
&3 + \overline{1}\\
&\\
&2 + \overline{1} + 1\\
&
\end{align*}
\end{flushleft}
\end{minipage}
\end{center}
\vspace{0.3cm}
The function that counts the number of overpartitions of an integer $n$ is denoted by $\ov(n)$ and is called the \textit{overpartition function}. Thus, the previous example shows that $\ov(4) = 14$.

As Corteel and Lovejoy explain in \cite{CL04}, overpartitions constitute important combinatorial structures in the study of $q$-series identities and hypergeometric series. In particular, they have been used to give combinatorial interpretations for different $q$-series identities. For example, some works in this direction are \cite{hirschhorn2005arithmetic} and \cite{lovejoy2008rank}.

Overpartitions have also been studied in relation to topics of great interest in number theory and the theory of integer partitions, like ranks, cranks and mock theta functions, for example in \cite{BLO09}, \cite{ADSY17}, \cite{Lin20} and \cite{Zha21}, just to cite a few.
More recently, overpartitions have been linked directly to the number of partitions of an integer into different parts in \cite{Mer22}.

As in the case of the partition function, another important line of research in the study of the overpartition function is its divisibility properties. For example, it has been found that $\overline{p}(n)$ satisfies distinct congruences of Ramanujan-type similar to the ones that are known for the partition function $p(n)$; see, for example \cite{Treneer06}, \cite{Treneer08}, \cite{Chen}, \cite{Xia}, and \cite{RSST21}.


As for partitions, the generating function for overpartitions can be expressed as an infinite product, namely
\begin{equation}
    \label{eqn:etaq}   
	\sum_{n\geq 0} \ov(n)q^n =
	\prod_{n \geq 1}\frac{1-q^{2n}}{\left(1-q^n\right)^2}.
\end{equation}
Given a positive integer $n_0$, by expanding this infinite product with the adequate precision, we can obtain the values of $\ov(n)$ for $1 \leq n \leq n_0$.
This naive method is computationally expensive, and does not allow to compute individual, isolated values of $\ov(n)$ easily; the same holds when using the recursive formula which we give in Proposition~\ref{prop:recursive}.
The main goal of this article is to give an efficient algorithm for computing $\ov(n)$, analogous to the one given by Johansson in \cite{Joh12} for the partition function.  

\medskip

The main idea behind Johannson's article is to use the Hardy-Rademacher-Ramanujan formula, which gives $p(n)$ as a convergent infinite series.
By computing enough terms of this series with the adequate precision and rounding to the nearest integer he obtains $p(n)$.
To achieve this efficiently, he gives a simple formula for each term.

We replicate this process in the context of overpartitions. Firstly, we state a Hardy-Rademacher-Ramanujan formula which, leaving the details for Section~\ref{sect:HRR}, gives that
\begin{align}
\label{eqn:HRRintro}
\ov(n) = \frac{1}{4n} \sum_{\substack{k \geq 1\\ 2 \nmid k}} \frac{1}{\sqrt{k}}\, \A_k(n)\, U\left( \frac{\ct(n)}{k} \right).
\end{align}
Then we give in Theorem~\ref{thm:explicit-error-bound} a bound for the error obtained when truncating this series.
Finally, to make our calculations viable, we need to be able to compute efficiently its terms.
In order to accomplish this, we show in Theorem~\ref{Multiplicativity} that $\A_k(n)$ is, in a particular sense, multiplicative in $k$, using classical tools like Dedekind sums and results due to Rademacher and Whiteman.
Then we show in Theorem~\ref{P-Power}, using a result due to Salié, that when $q$ is a prime power $\A_q(n)$ can be computed by a simple formula.
These two theorems are proved in Section~\ref{sect:proofs}.

All of the above results are summarized in Algorithm~\ref{alg:overpartitions}, which computes $\ov(n)$.
While we do not focus on optimizing our code as Johansson does, despite the fact that one could, we do point out that we can calculate $\ov(10^{14})$: this is a number close to $4.31 \cdot 10^{1363748}$
(see Proposition~\ref{prop:growth}), whose last 50 digits are
\[
18854845964512314768846736319878009378857016552454.
\]
It is stored in \cite[v1.0]{code}.

We finally turn our attention to the application that motivated us from the beginning: to obtain explicit instances of the infinite families of congruences satisfied by $\ov(n)$ according to \cite[Proposition 1.5]{Treneer08}.
For this purpose we need to implement custom methods to compute the summands in \eqref{eqn:HRRintro}, which are real numbers, modulo small integers; otherwise the numbers get so big that the arithmetic operations slow down considerably and require an excessive amount of memory.
This is accomplished in Algorithm~\ref{alg:overpartitionsmodm}.
Based on the computations we carried out, we find over a hundred new congruences satisfied by $\ov(n)$, as stated in Theorems~\ref{thm:congruences1} and \ref{thm:congruences2}.
In particular we exhibit the first congruences modulo $7$, as far as we know.

\bigskip


\subsection*{Acknowledgments} The third author thanks the Fulbright Foundation for supporting him during his time in Costa Rica, without which this project would not have happened. The first author thanks the Centro de Investigación en Matemática Pura y Aplicada and the School of Mathematics of the University of Costa Rica for administrative help and support during this project.

\section{A recursive formula for the values of \texorpdfstring{$\ov(n)$}{pbar(n)}}
\label{sect:recursive}

The partition function can be computed recursively using pentagonal numbers (see, e.g. \cite[(1.2)]{Joh12}).
In this section we show that overpartitions satisfy a simpler recursion.

\begin{proposition}
\label{prop:recursive}
    Define $\ov{(0)} = 1$ and $\ov{(n)} = 0$ for every $n < 0$. Then for every $n \geq 1$ we have the recursive formula
    \begin{align}\label{Overpartition-recurrence}
        \ov(n)=2\sum_{k=1}^{\infty}(-1)^{k+1}\ov(n-k^2).
    \end{align}
\end{proposition}

\begin{proof}
Consider the generating function of the overpartition function
$$\overline{f}(q) := \displaystyle{\sum_{n = 0}^{\infty} \ov(n)q^n}.$$
Using \eqref{eqn:etaq} and \cite[(2.2.12)]{Andrews} we get that
\begin{align}
\frac{1}{\overline{f}(q)} 
= \sum_{n = -\infty}^{\infty} (-1)^n q^{n^2} = 1 + 2\sum_{n = 1}^{\infty}(-1)^{n}q^{n^2}.
\end{align}
Then we can write
\begin{align*}
\frac{1}{\overline{f}(q)}
= \sum_{k = 0}^{\infty} b_k q^k,
\end{align*}
where 
\begin{align}
b_k = 
\begin{cases}\label{bk-definition}
1, & \text{if $k = 0$,}\\
2(-1)^n, & \text{if $k = n^2$ for some $n \in \Z_{\geq 1}$, and}\\
0, & \text{if $k$ is not a perfect square}.
\end{cases}
\end{align}
Hence, multiplying the series for $\overline{f}(q)$ and $1/\overline{f}(q)$ we obtain
\begin{align}\label{Series-product}
1 = \overline{f}(q) \cdot \frac{1}{\overline{f}(q)} = \left( \sum_{n = 0}^{\infty} \ov(n)q^n \right) \left( \sum_{k = 0}^{\infty} b_k q^k \right)=\sum_{n=0}^{\infty}c_{n}q^{n},
\end{align}
where 
\begin{align}\label{Cauchy-coefficients}
    c_{n}=\sum_{k=0}^{n}\ov(n-k)b_{k}.
\end{align}
Now, comparing coefficients in (\ref{Series-product}), we see that $c_0 = 1$ and $c_{n}=0$ for all $n\geq 1$. Moreover, solving for $\ov(n)$ in (\ref{Cauchy-coefficients}), i.e. solving for the $k = 0$ term, using the definition of $b_k$ given in (\ref{bk-definition}), and defining $\ov(0)=1$ and $\ov(n)=0$ if $n<0$, we obtain the recurrence relation
\eqref{Overpartition-recurrence},
valid for $n \geq 1$. 
\end{proof}

\begin{remark}
    We observe that although the fact that $\overline{p}(n)$ is even for every $n \geq 1$ can be seen combinatorially from the definition,
    the appearance of the factor $2$ in the recursive formula (\ref{Overpartition-recurrence}) clearly shows this fact.
\end{remark}

\section{A Hardy-Ramanujan-Rademacher type formula for overpartitions}
\label{sect:HRR}

The partition function can be given as a convergent infinite series.
More precisely, Hardy-Ramanujan and Rademacher showed in \cite{HR18} and \cite{R38} that
\begin{align}
\label{HRR}
p(n) =
    \frac{4}{24n-1}
    \sum_{k \geq 1} \sqrt{\tfrac{3}{k}} \,
    A_k(n) \,
    U\left(\frac{C(n)}{k}\right).
\end{align}
Here we denote
\begin{align}
A_k(n) := \sum_{\substack{0 \leq h < k\\ (h, k) = 1}}
    \omega(h, k) \, 
    e^{-2\pi i n h/k},
\end{align}
where $\omega(h, k) := e^{\pi i s(h, k)}$,
and $s(h, k)$ denotes the Dedekind sum
\begin{align}
s(h, k) := \sum_{r = 1}^{k - 1} \frac{r}{k} \left( \frac{hr}{k} - \left \lfloor \frac{hr}{k} \right \rfloor - \frac{1}{2} \right).
\end{align}
Furthermore, we denote
\begin{align*}
U(x) & := \cosh{x} - \frac{\sinh{x}}{x}, \\
C(n) & := \frac{\pi}6 \sqrt{24n-1}.
\end{align*}

Similarly, in relation to the overpartition function $\overline{p}(n)$, in 1939 Zuckerman \cite[Equations (8.36) and (8.53)]{Zuc39} gave explicit formulas for the Fourier coefficients of the generating function of the overpartition function. In particular, his formulas show that
\begin{align}\label{Zuckerman}
\ov(n) = \frac{1}{2\pi} \sum_{\substack{k \geq 1\\ 2 \nmid k}} \sqrt{k} \widetilde{A}_k(n) \frac{d}{dn}\left( \frac{\sinh{\left( \frac{\pi \sqrt{n}}{k} \right)}}{\sqrt{n}} \right),
\end{align}
where we let
\begin{align}\label{Aktilde}
\widetilde{A}_k(n) := \sum_{\substack{0 \leq h < k\\ (h, k) = 1}}
    \frac{\omega(h, k)^2}{\omega(2h, k)}\,
    e^{-2\pi i n h/k}.
\end{align}

Now a simple calculation shows that 
\begin{align*}
\frac{d}{dn}\left( \frac{\sinh{\left( \frac{\pi \sqrt{n}}{k} \right)}}{\sqrt{n}} \right) = \frac{\pi}{2kn} U\left( \frac{\ct(n)}{k} \right),
\end{align*}
where $\ct(n) := \pi \sqrt{n}$.
Therefore, the Hardy-Ramanujan-Rademacher type formula (\ref{Zuckerman}) can be rewritten more succinctly, and similarly to \eqref{HRR}, as
\begin{align}\label{Overpartition-Rademacher}
\ov(n) = \frac{1}{4n} \sum_{\substack{k \geq 1\\ 2 \nmid k}} \frac{1}{\sqrt{k}}\, \A_k(n)\, U\left( \frac{\ct(n)}{k} \right).
\end{align}
This is the formula that we will use along this article.

\section{Truncating the series for \texorpdfstring{$\ov(n)$}{pbar(n)} and an explicit bound for the error}

In this section we will prove an explicit bound for the error term $\Rt(n, N)$ obtained when truncating the infinite series (\ref{Overpartition-Rademacher}). The explicit bound we prove below in Theorem \ref{thm:explicit-error-bound} will then be used in Algorithm \ref{alg:overpartitions}.  Strictly speaking, our implementations could use bounds due to Engel \cite{Eng17}, but the bound we prove below makes our computations more efficient.

Thus, let $\Rt(n, N)$ be defined by 
\begin{align}\label{R-Definition1}
\ov(n) = \frac{1}{4n} \sum_{\substack{k = 1\\ 2 \nmid k}}^N \frac{1}{\sqrt{k}} \,\A_k(n) \,U\left( \frac{\ct(n)}{k} \right) + \Rt(n, N),
\end{align}
that is, 
\begin{align}\label{R-series}
    \Rt(n, N) := \frac{1}{4n} \sum_{\substack{k = N+1\\ 2 \nmid k}}^{\infty} \frac{1}{\sqrt{k}} \,\A_k(n) \,U\left( \frac{\ct(n)}{k} \right).
\end{align}

Recently, Engel \cite[Theorem 2.3]{Eng17} proved that the error term $\Rt(n, N)$ is bounded by
\begin{align}
|\Rt(n, N)| \leq \frac{N^{5/2}}{\pi n^{3/2}} \sinh{\left( \frac{\pi \sqrt{n}}{N} \right)}.
\end{align}
Now, as can be checked by using the asymptotic formula $\sinh{(x)} = x(1 + o(1))$ as $x \to 0$,
we have that
$$
\frac{N^{5/2}}{\pi n^{3/2}} \sinh{\left( \frac{\pi \sqrt{n}}{N} \right)} \to \infty \quad \text{as $N \to \infty$}.
$$
This shows that Engel's bound is rather weak since we know that $\Rt(n, N) \to 0$ as $N \to \infty$.

Instead, by employing a method similar to the one used by Rademacher in \cite{R38}, we will now prove a stronger bound for the error term $\Rt(n, N)$.
Note the similarity between \eqref{eqn:error_bound} and the bound (8.3) in \cite{R38}.

\begin{theorem}\label{thm:explicit-error-bound}
    Let $n \geq 1$ be a fixed positive integer. Then for every $N \geq 1$, the error term $\Rt(n, N)$ satisfies the upper bound
\begin{align}\label{eqn:error_bound}
    |\Rt(n, N)| \leq M(n, N),
\end{align}
where 
\begin{align}\label{eqn:error-bound-definition}
    M(n, N) := \frac{1}{4\pi} \left( \frac{N+1}{n} \right)^{3/2} \left( \frac{\pi \sqrt{n}}{N+1} \cosh{ \left( \frac{\pi \sqrt{n}}{N+1} \right) } + (2N + 1) \sinh{ \left( \frac{\pi \sqrt{n}}{N+1} \right) } - 2\pi \sqrt{n} \right).
\end{align}
Moreover, the bound $M(n, N)$ satisfies the asymptotic formula
\begin{align}
    \label{eqn:error_bound_asymptotic}
    M(n, N) = \frac{\pi^2}{12} \frac{1}{\sqrt{N + 1}} (1 + o(1)) \quad \text{as $N \to \infty$}.
\end{align}
\end{theorem}

\begin{proof}
We use the formula 
$$
\Rt(n, N) = \frac{1}{2\pi} \sum_{\substack{k = N + 1\\ 2 \nmid k}}^{\infty}\sqrt{k} \widetilde{A}_k(n) \frac{d}{dn}\left( \frac{\sinh{\left( \frac{\pi \sqrt{n}}{k} \right)}}{\sqrt{n}} \right)
$$
for the error term. We will now bound $\widetilde{A}_k(n)$ and the derivative term. Then, we will bound the resulting series by using the triangle inequality.

Now, observe that trivially we have
$$
\left| \widetilde{A}_k(n) \right| = \left| \sum_{\substack{0 \leq h < k\\ (h, k) = 1}} \frac{\omega(h, k)^2}{\omega(2h, k)}\, e^{-2\pi i n h/k} \right| \leq \sum_{\substack{0 \leq h < k\\ (h, k) = 1}} 1 = \varphi(k) \leq k.
$$
Next, for the derivative we note that using the Taylor expansion at $x = 0$ for $\sinh{(x)}$, we have
$$
\frac{\sinh{\left( \frac{\pi \sqrt{n}}{k} \right)}}{\sqrt{n}} = \sum_{v = 0}^{\infty} \left( \frac{\pi}{k} \right)^{2v + 1} \frac{n^v}{(2v+1)!}.
$$
Therefore, differentiating this and substituting back in the formula for $\Rt(n, N)$ we have
\begin{align*}
    |\Rt(n, N)| &\leq \frac{1}{2\pi} \sum_{\substack{k = N + 1\\ 2 \nmid k}}^{\infty}\sqrt{k} \left| \widetilde{A}_k(n) \right|  \left|\frac{d}{dn}\left( \frac{\sinh{\left( \frac{\pi \sqrt{n}}{k} \right)}}{\sqrt{n}} \right) \right| \\
    &= \frac{1}{2\pi} \sum_{\substack{k = N + 1\\ 2 \nmid k}}^{\infty} k^{3/2}  \sum_{v = 1}^{\infty} \left( \frac{\pi}{k} \right)^{2v + 1} \frac{vn^{v - 1}}{(2v+1)!} \\
    &= \frac{1}{2\pi} \sum_{v = 1}^{\infty} \frac{v \pi^{2v + 1}}{(2v + 1)!} n^{v - 1} \sum_{\substack{k = N + 1\\ 2 \nmid k}}^{\infty} \frac{1}{k^{2v - 1/2}}.
\end{align*}
Now, note that we have
\begin{align*}
    \sum_{\substack{k = N + 1\\ 2 \nmid k}}^{\infty} \frac{1}{k^{2v - 1/2}} = \sum_{k = \left \lceil \frac{N}{2} \right \rceil}^{\infty} \frac{1}{(2k+1)^{2v - 1/2}}.
\end{align*}
Now, recall that the integral test for convergence says that if $f \colon [T, \infty) \to \R$ is a monotone decreasing function, with $T \in \Z_{\geq 1}$, then 
$$
\sum_{k = T}^{\infty} f(k) \leq f(T) + \int \limits_{T}^{\infty} f(x) \, dx,
$$
provided that the integral converges. Then, applying this to the function $f(x) := \dfrac{1}{(2x + 1)^{2v - 1/2}}$, we get
\begin{align*}
    \sum_{k = \left \lceil \frac{N}{2} \right \rceil}^{\infty} \frac{1}{(2k+1)^{2v - 1/2}} &\leq \frac{1}{\left( 2 \left \lceil \frac{N}{2} \right \rceil + 1 \right)^{2v - 1/2}} + \int \limits_{ \left \lceil \frac{N}{2} \right \rceil }^{\infty} \dfrac{1}{(2x + 1)^{2v - 1/2}} \, dx \\
    &= \frac{1}{\left( 2 \left \lceil \frac{N}{2} \right \rceil + 1 \right)^{2v - 1/2}} + \frac{1}{4v - 3} \frac{1}{\left( 2 \left \lceil \frac{N}{2} \right \rceil + 1 \right)^{2v - 3/2}} \\
    &\leq \frac{1}{(N+1)^{2v - 1/2}} + \frac{1}{4v - 3} \frac{1}{(N+1)^{2v - 3/2}}.
\end{align*}
Thus, we let 
$$
A(v, N) := \frac{1}{(N+1)^{2v - 1/2}} \quad \text{and} \quad B(v, N) := \frac{1}{4v - 3} \frac{1}{(N+1)^{2v - 3/2}}.
$$

Then we have that
\begin{align}
    |\Rt(n, N)| &\leq \frac{1}{2\pi} \sum_{v = 1}^{\infty} \frac{v \pi^{2v + 1}}{(2v + 1)!} n^{v - 1} \sum_{k = \left \lceil \frac{N}{2} \right \rceil}^{\infty} \frac{1}{(2k+1)^{2v - 1/2}} \nonumber \\
    &= \frac{1}{2\pi} \sum_{v = 1}^{\infty} \frac{v \pi^{2v + 1}}{(2v + 1)!} n^{v - 1} \left( A(v, N) + B(v, N) \right) \nonumber \\
    &= \frac{1}{2\pi} \sum_{v = 1}^{\infty} \frac{v \pi^{2v + 1}}{(2v + 1)!} n^{v - 1} A(v, N) + \frac{1}{2\pi} \sum_{v = 1}^{\infty} \frac{v \pi^{2v + 1}}{(2v + 1)!} n^{v - 1} B(v, N).
    \label{eqn:two_series}
\end{align}
We now analyze separately each of the two series in the last expression.

For the first series in (\ref{eqn:two_series}) we have
\begin{align}\label{eqn:error-A-bound}
    \sum_{v = 1}^{\infty} \frac{v \pi^{2v + 1}}{(2v + 1)!} n^{v - 1} A(v, N) &= \sum_{v = 1}^{\infty} \frac{v \pi^{2v + 1}}{(2v + 1)!} n^{v - 1} \frac{1}{(N+1)^{2v - 1/2}} \notag \\
    &= \left( \frac{N+1}{n} \right)^{3/2} \sum_{v = 1}^{\infty} \frac{v}{(2v+1)!} \left( \frac{\pi \sqrt{n}}{N+1} \right)^{2v + 1} \notag \\
    &= \frac{1}{2} \left( \frac{N+1}{n} \right)^{3/2} H \left( \frac{\pi \sqrt{n}}{N+1} \right),
\end{align}
where $H(x):= x \cosh{x} - \sinh{x}$ satisfies 
$$
H(x) = x^2 \frac{d}{dx} \left( \frac{\sinh{x}}{x} \right) = 2\sum_{v = 1}^{\infty} \frac{v}{(2v + 1)!} x^{2v + 1}
$$
for every $x \in \R$.

Similarly, for the second series in \eqref{eqn:two_series}, we have
\begin{align}\label{eqn:error-B-bound}
    \sum_{v = 1}^{\infty} \frac{v \pi^{2v + 1}}{(2v + 1)!} n^{v - 1} B(v, N) &= \sum_{v = 1}^{\infty} \frac{v \pi^{2v + 1}}{(2v + 1)!} n^{v - 1} \frac{1}{4v - 3} \frac{1}{(N+1)^{2v - 3/2}} \notag \\
    &= \sum_{v = 1}^{\infty} \frac{v}{4v - 3} \frac{ \pi^{2v + 1}}{(2v + 1)!} n^{v - 1} \frac{1}{(N+1)^{2v - 3/2}} \notag \\
    &\leq \sum_{v = 1}^{\infty} \frac{\pi^{2v + 1} (\sqrt{n})^{2v - 2}}{(2v + 1)! (N+1)^{2v - 3/2}} \notag \\
    &= \frac{(N+1)^{5/2} }{n^{3/2}} \sum_{v = 1}^{\infty} \frac{1}{(2v+1)!} \left( \frac{\pi \sqrt{n}}{N+1} \right)^{2v + 1} \notag \\
    &=  \frac{(N+1)^{5/2} }{n^{3/2}} \left( \sinh{ \left( \frac{\pi \sqrt{n}}{N+1} \right) } - \frac{\pi \sqrt{n}}{N+1}   \right).
\end{align}

Thus, combining \eqref{eqn:two_series}, \eqref{eqn:error-A-bound} and \eqref{eqn:error-B-bound}, we get

\begin{align*}
    &|\widetilde{R}(n, N)| \leq \frac{1}{2\pi} \frac{1}{2} \left( \frac{N+1}{n} \right)^{3/2} H \left( \frac{\pi \sqrt{n}}{N+1} \right) + \frac{1}{2\pi} \frac{(N+1)^{5/2} }{n^{3/2}} \left( \sinh{ \left( \frac{\pi \sqrt{n}}{N+1} \right) } - \frac{\pi \sqrt{n}}{N+1}   \right) \\
    &= \frac{1}{4\pi} \left( \frac{N+1}{n} \right)^{3/2} \left( \frac{\pi \sqrt{n}}{N+1} \cosh{ \left( \frac{\pi \sqrt{n}}{N+1} \right)} - \sinh{\left( \frac{\pi \sqrt{n}}{N+1}\right)} + 2(N+1) \sinh{\left( \frac{\pi \sqrt{n}}{N+1}\right)} - 2\pi\sqrt{n} \right) \\
    &= \frac{1}{4\pi} \left( \frac{N+1}{n} \right)^{3/2} \left( \frac{\pi \sqrt{n}}{N+1} \cosh{ \left( \frac{\pi \sqrt{n}}{N+1} \right)} + (2N+1) \sinh{\left( \frac{\pi \sqrt{n}}{N+1}\right)} - 2\pi\sqrt{n} \right).
\end{align*}
This proves \eqref{eqn:error_bound}.

Now we will analyze the asymptotic growth of $M(n,N)$ as $N \to \infty$.
Observe that 
\begin{align*}
    \cosh(x) = 1 + \frac{x^2}{2} + x^2 o(1) \quad \text{and} \quad \sinh(x) = x + \frac{x^3}{6} + x^3 o(1)
\end{align*}
as $x \to 0$. Therefore, this implies that
\begin{align*}
    \cosh{ \left( \frac{\pi \sqrt{n}}{N+1} \right) } & = 1 + \frac{1}{2} \left( \frac{\pi \sqrt{n}}{N+1} \right)^2 + \left( \frac{\pi \sqrt{n}}{N+1} \right)^2 o(1), \qquad \text{and}\\
    \sinh{ \left( \frac{\pi \sqrt{n}}{N+1} \right) } & = \frac{\pi \sqrt{n}}{N+1}  + \frac{1}{6} \left( \frac{\pi \sqrt{n}}{N+1} \right)^3 + \left( \frac{\pi \sqrt{n}}{N+1} \right)^3 o(1)
\end{align*}
as $N\to\infty$. Using these expansions we see that
\begin{align}\label{eqn:error-proof-asymptotic}
    &\qquad \frac{\pi \sqrt{n}}{N+1} \cosh{ \left( \frac{\pi \sqrt{n}}{N+1} \right) } + (2N+1) \sinh{ \left( \frac{\pi \sqrt{n}}{N+1} \right) } - 2\pi \sqrt{n} \notag \\
    &= \frac{\pi \sqrt{n}}{N+1} \left( 1 + \frac{1}{2} \left( \frac{\pi \sqrt{n}}{N+1} \right)^2 + \left( \frac{\pi \sqrt{n}}{N+1} \right)^2 o(1) \right) \notag \\
    &+ (2N+1) \left(  \frac{\pi \sqrt{n}}{N+1}  + \frac{1}{6} \left( \frac{\pi \sqrt{n}}{N+1} \right)^3 + \left( \frac{\pi \sqrt{n}}{N+1} \right)^3 o(1) \right) - 2\pi \sqrt{n} \notag \\
    &= \frac{1}{2} \left( \frac{\pi \sqrt{n}}{N+1} \right)^3 + \left( \frac{\pi \sqrt{n}}{N+1} \right)^3 o(1) \ + \frac{2N+1}{6} \left( \frac{\pi \sqrt{n}}{N+1} \right)^3 + (2N+1) \left( \frac{\pi \sqrt{n}}{N+1} \right)^3 o(1) \notag \\
    &= \frac{\pi^3 n^{3/2}}{3} \frac{1}{(N+1)^2} + \frac{1}{(N+1)^2} o(1),
\end{align}
where we have used the facts that 
$$
\frac{\pi \sqrt{n}}{N+1} + (2N + 1) \frac{\pi \sqrt{n}}{N+1} - 2\pi \sqrt{n} = 0 \quad \text{and} \quad \dfrac{1}{(N+1)^3} = \dfrac{1}{(N+1)^2} o(1)
$$
as $N \to \infty$. Moreover, in obtaining the expression (\ref{eqn:error-proof-asymptotic}) we wrote $2N+1 = 2(N+1) - 1$ and distributed the products in order to simplify the formula.

Hence, using the expansion (\ref{eqn:error-proof-asymptotic}) in the definition of $M(n, N)$ we get that
\begin{align*}
    M(n, N) &= \frac{1}{4\pi} \left( \frac{N+1}{n} \right)^{3/2} \left( \frac{\pi^3 n^{3/2}}{3} \frac{1}{(N+1)^2} + \frac{1}{(N+1)^2} o(1) \right) \\
    &=\frac{\pi^2}{12} \frac{1}{\sqrt{N+1}}(1 + o(1))
\end{align*}
as $N \to \infty$. 
\end{proof}

\section{The order of growth of \texorpdfstring{$\ov(n)$}{pbar(n)}}

As was mentioned in the introduction, it is known that the order of growth of the partition function $p(n)$ is given by
$$
p(n) \sim \frac{1}{4n\sqrt{3}} e^{\pi \sqrt{2n/3}} \quad \text{as $n \to \infty$.}
$$
In this section, using the bound for the error term that we proved in Section 3, we compute the order of growth of the overpartition function.

\begin{proposition}
\label{prop:growth}
The overpartition function $\overline{p}(n)$ has order of growth
\begin{align}\label{Growth}
\ov(n) \sim \frac{e^{\pi \sqrt{n}}}{8n}
\end{align}
as $n \to \infty$
\end{proposition}

\begin{proof}
We will obtain the order of growth of the overpartition function from the infinite series (\ref{Overpartition-Rademacher}). In fact, as we will see, the dominant term is the first one in the series. Thus, taking $N = 1$ in formula (\ref{R-Definition1}) we get
\begin{align}
\ov(n) = \frac{1}{4n} \A_1(n) U(\pi \sqrt{n}) + \Rt(n, 1).
\end{align}
Now, from equation (\ref{Aktilde}), which gives the definition of $\widetilde{A}_k(n)$, we see that $\A_1(n) = 1$ for any $n \in \Z$. Also, note that
\begin{align}\label{U-formula}
U(x) = \cosh{x} - \frac{\sinh{x}}{x} = \frac{e^x}{2}\left( 1 + e^{-2x} + \frac{e^{-2x} - 1}{x} \right), 
\end{align}
and hence this shows that $U(x) \sim \dfrac{e^x}{2}$ as $x \to +\infty$. Therefore, the leading term satisfies
$$
\frac{1}{4n} \A_1(n) U(\pi \sqrt{n}) \sim \frac{e^{\pi \sqrt{n}}}{8n}
$$
as $n \to \infty$. Then, using this we have
\begin{align}\label{Order-Calculation}
\ov(n) \cdot \left( \frac{e^{\pi \sqrt{n}}}{8n} \right)^{-1} = \frac{2}{e^{\pi \sqrt{n}}} U(\pi \sqrt{n}) + \frac{8n}{e^{\pi \sqrt{n}}} \Rt(n, 1).
\end{align}
Now, for the first term in (\ref{Order-Calculation}), we know from \eqref{U-formula} that
\begin{align}\label{Order-Main-Term}
\frac{2}{e^{\pi \sqrt{n}}} U(\pi \sqrt{n}) \xrightarrow[n \to \infty]{} 1.
\end{align}

Next, from Theorem \ref{thm:explicit-error-bound} we know that $|\Rt(n, 1)| \leq M(n, 1)$, where $M(n, N)$ is defined in equation (\ref{eqn:error-bound-definition}). Then, a short calculation shows that
$$
M(n, 1) = \frac{\sqrt{2}}{8n} e^{\pi\sqrt{n}/2} \left( 1 + e^{-\pi\sqrt{n}} + \frac{6(1 - e^{-\pi\sqrt{n}})}{\pi \sqrt{n}} - 8e^{-\pi \sqrt{n}/2} \right).
$$
Hence, we see that
\begin{align}\label{eqn:order-of-growth-error}
\frac{8n}{e^{\pi \sqrt{n}}} |\Rt(n, 1)| &\leq \frac{8n}{e^{\pi \sqrt{n}}} M(n, 1) \notag \\
&= \sqrt{2} e^{-\pi\sqrt{n}/2} \left( 1 + e^{-\pi\sqrt{n}} + \frac{6(1 - e^{-\pi\sqrt{n}})}{\pi \sqrt{n}} - 8e^{-\pi \sqrt{n}/2} \right) \xrightarrow[n \to \infty]{} 0.
\end{align}

Therefore, combining (\ref{Order-Calculation}), (\ref{Order-Main-Term}) and (\ref{eqn:order-of-growth-error}), we conclude that 
\begin{align*}
\lim_{n \to \infty} \ov(n) \cdot \left( \frac{e^{\pi \sqrt{n}}}{8n} \right)^{-1} = 1,
\end{align*}
and this proves the proposition.
\end{proof}



\section{Arithmetic properties of \texorpdfstring{$\A_k(n)$}{Aktilde(n)} and its efficient evaluation}\label{AktildeSection}

In this section we study the function $\A_k(n)$ with the goal of using its properties to simplify and speed up our method to compute the values of the overpartition function $\ov(n)$. Since the series (\ref{Overpartition-Rademacher}) is a sum over odd values of $k$, we will only be interested in properties of $\A_k(n)$ for odd $k$.

More precisely, our first result, Theorem \ref{Multiplicativity}, will allow us to prove that if $k$ is an odd positive integer with prime factorization $k = p_1^{\alpha_1} \cdots p_j^{\alpha_j}$, then for a positive integer $n$, there exist positive integers $n_1, \dots, n_j$ such that

\begin{align}
\A_{k}(n) = \A_{p_{1}^{\alpha_1}}(n_1) \cdots \A_{p_{j}^{\alpha_j}}(n_j).
\end{align}
This decomposition then shifts the problem of computing the values $\A_k(n)$ to computing the values $\A_{q}(n)$ when $q$ is the power of an odd prime. This is achieved in Theorem~\ref{P-Power}.

\subsection{A multiplicativity relation satisfied by \texorpdfstring{$\A_k(n)$}{Aktilde(n)}}

In \cite[Theorem 1]{Leh38},  D. H. Lehmer proved that the function $A_k(n)$ that appears in the Hardy-Ramanujan-Rademacher series (\ref{HRR}) for the partition function $p(n)$ satisfies a certain multiplicativity relation. More precisely,  Lehmer proved that if $k_1$ and $k_2$ are positive integers with $\gcd(k_1,k_2) = \gcd(k_1 k_2,6) = 1$, then, given a positive integer $n$, we have that
$$
A_{k_1 k_2}(n) = 
A_{k_1}(n_1) A_{k_2}(n_2)
$$
if $n_1,n_2$ are positive integers satisfying the congruences
$$
24 n \equiv k_2^2(n_1 - 1 ) + 1 \pmod{k_1},
\qquad
24 n \equiv k_1^2(n_2 - 1 ) + 1 \pmod{k_2}.
$$

Interestingly, we found that the function $\A_k(n)$ defined in (\ref{Aktilde}) satisfies a very similar type of multiplicativity relation, which we state in the following theorem.


\begin{theorem}\label{Multiplicativity}
Let $k_1$ and $k_2$ be coprime positive integers. 
Given a positive integer $n$, let $n_1,n_2$ be positive integers satisfying
$$
k_2^2 n_1 \equiv n \pmod{k_1},
\qquad
k_1^2 n_2 \equiv n \pmod{k_2}.
$$
Then we have
$$
\A_{k_1 k_2}(n) = \A_{k_1}(n_1) \A_{k_2}(n_2).
$$
\end{theorem}

\subsection{The prime power case}
 We carried out an analysis similar to Lehmer's but for the function $\A_q(n)$ when $q$ is a prime power, and arrived at compact formulas.

\begin{theorem}\label{P-Power}
Let $p \geq 3$ be a prime and let $q = p^{\alpha}$ with $\alpha  \geq 1$.
Let $n$ be a positive integer.
Then we have
\begin{align}
\label{eqn:P-power}
\A_q(n) = 
\begin{cases}
\sqrt{q}, & \text{if $p \mid n$ and $\alpha = 1$}, \\
0, & \text{if $p \mid n$ and $\alpha > 1$},\\
0, & \text{if $p\nmid n$ and $-n$ is not a quadratic residue modulo $q$, and}\\
2\sqrt{q} \cos{\left( \dfrac{4\theta\pi}{q} \right)},&
    \text{if $p\nmid n$ and $-n \equiv (4\theta)^2 \pmod{q}$ for some $\theta \in \Z$}.
\end{cases}
\end{align}
\end{theorem}


\section{Proofs of Theorems \ref{Multiplicativity} and \ref{P-Power}}
\label{sect:proofs}

We start by setting up the notation and preliminary results that will be used in both proofs.
First, recall that 
\begin{align}
\label{Aktilde_bis}
\widetilde{A}_k(n) := \sum_{\substack{0 \leq h < k\\ (h, k) = 1}} \frac{\omega(h, k)^2}{\omega(2h, k)} e^{-2\pi i n h/k},
\end{align}
with $\omega(h, k) = e^{\pi i s(h, k)}$ and $s(h,k)$ is the Dedekind sum given by
\begin{align}
s(h, k) = \sum_{r = 1}^{k - 1} \frac{r}{k} \left( \frac{hr}{k} - \left \lfloor \frac{hr}{k} \right \rfloor - \frac{1}{2} \right).
\end{align}
In the following proposition we gather some known properties of the Dedekind sums that will be important in our discussion.

\begin{proposition}[Properties of Dedekind sums]\label{DedekindSumProperties}
Let $h, k$ be positive integers with $\gcd{(h, k)} = 1$.
Then the Dedekind sum $s(h, k)$ satisfies the following properties.
\begin{enumerate}
\item The denominator of $s(h, k)$ is a divisor of $2k \gcd{(3, k)}$. In particular $6k\,s(h, k)$ is always an integer.
\label{item:s-entero}
\item If $\theta := \gcd{(k, 3)}$, then
$
12hk\,s(h, k) \equiv h^2 + 1 \pmod{\theta k}.
$
\label{item:s-h2}
\item $12k\,s(h, k) \equiv 0 \pmod{3}$ if and only if $3 \nmid k$.
\label{item:s-mod3}
\item If $k$ is odd, then $\displaystyle{12k\,s(h, k) \equiv k + 1 - 2 \legendre{h}{k} \pmod{8}.}$
\label{item:s-mod8}
\item If $h \equiv h' \pmod{k}$, then $s(h, k) = s(h', k)$.
\label{item:s-hhp}
\end{enumerate}
\end{proposition}

\begin{proof}
For (\ref{item:s-entero}) see \cite[Chapter 3, Theorem 2]{RG72}  or \cite[Theorem 3.8]{Apo90}. For properties (\ref{item:s-h2}) and (\ref{item:s-mod3}) see \cite[Theorem 3.8]{Apo90} and its proof. For (\ref{item:s-mod8}) see \cite[(42)]{RG72}, and for (\ref{item:s-hhp}) see \cite[Theorem 3.6 (a)]{Apo90}.
\end{proof}



\subsection{Proof of Theorem \ref{Multiplicativity}}

We start by stating the following fundamental result, due to Radema\-cher and Whiteman (\cite[Theorem 20]{RW41}).

\begin{theorem}\label{Rademacher-Whiteman}
Let $a, b, c$ be pairwise coprime positive integers such that $24 \mid abc$. Then the Dedekind sums satisfy the congruence
\begin{align}
\left( s(ab, c) - \frac{ab}{12c} \right) + \left( s(bc, a) - \frac{bc}{12a} \right) - \left( s(b, ac) - \frac{b}{12ac} \right) \equiv 0 \pmod{2}.
\end{align}
\end{theorem}

With this setup we can now give the proof.

\begin{proof}[Proof of Theorem~\ref{Multiplicativity}]
First, if $n, k, h, h' \in \Z$ are integers with $k \geq 1$ and $\gcd{(h, k)} = 1$, then by Proposition \ref{DedekindSumProperties} (\ref{item:s-hhp}) we have that if $h \equiv h' \pmod{k}$, then
\begin{align*}
\frac{\omega(h, k)^2}{\omega(2h, k)} e^{-2\pi i n h/k} = \frac{\omega(h', k)^2}{\omega(2h', k)} e^{-2\pi i n h'/k}.
\end{align*}
Therefore, for any set of representatives for $(\Z/k\Z)^\times$ we can write
\begin{align*}
\A_k(n) &= \sum_{\substack{[h] \in (\Z/k\Z)^{\times}}} \frac{\omega(h, k)^2}{\omega(2h, k)} e^{-2\pi i n h/k} \\
&= \sum_{\substack{[h] \in (\Z/k\Z)^{\times}}} \exp{\left( \pi i \left( 2s(h, k) - s(2h, k) - \frac{2n h}{k} \right) \right)}.
\end{align*}

Now, let $k_1, k_2$ be integers coprime, odd, positive integers and $n_1, n_2 \in \Z$. 
Then since  $(\Z/k_1\Z)^{\times} = \{ [k_2 h_1] \suchthat [h_1] \in (\Z/k_1\Z)^{\times} \}$ and $(\Z/k_2\Z)^{\times} = \{ [k_1 h_2] \suchthat [h_2] \in (\Z/k_2\Z)^{\times} \}$
we can write
\begin{multline*}
\A_{k_1}(n_1) \A_{k_2}(n_2) = 
\sum_{\substack{([h_1], [h_2]) \in (\Z/k_1\Z)^{\times}} \times (\Z/k_2\Z)^{\times}} \exp \left( \pi i \left( 2s(k_2h_1, k_1) - s(2k_2 h_1, k_1) - \frac{2n_1 k_2 h_1}{k_1} \right. \right. \\
\left. \left. - 2s(k_1 h_2, k_2) + s(2k_1 h_2, k_2) + \frac{2n_2 k_1 h_2}{k_2} \right) \right).
\end{multline*}
By the Chinese Remainder Theorem we know that $(\Z/k_1 k_2 \Z)^{\times} \simeq  (\Z/k_1 \Z)^{\times} \times (\Z/k_2 \Z)^{\times}$, where the isomorphism is given by $h + k_1k_2 \Z \mapsto (h + k_1\Z, h + k_2\Z)$. Thus, using this we can rewrite the previous equality as
\begin{multline*}
\A_{k_1}(n_1) \A_{k_2}(n_2) = 
\sum_{[h] \in (\Z/k_1k_2\Z)^{\times}} \exp \left( \pi i \left( 2s(k_2h, k_1) - s(2k_2 h, k_1) - \frac{2n_1 k_2 h}{k_1} \right. \right. \\
\left. \left. - 2s(k_1 h, k_2) + s(2k_1 h, k_2) + \frac{2n_2 k_1 h}{k_2} \right) \right).
\end{multline*}

On the other hand, we have
\begin{align*}
\A_{k_1 k_2}(n) = \sum_{\substack{[h] \in (\Z/k_1k_2\Z)^{\times}}} \exp{\left( \pi i \left( 2s(h, k_1 k_2) - s(2h, k_1 k_2) - \frac{2n h}{k_1 k_2} \right) \right)}.
\end{align*}
By the periodicity of $\exp{(z)}$, in order to prove that $\A_{k_1 k_2}(n) = \A_{k_1}(n_1) \A_{k_2}(n_2)$ it suffices to prove that
for every $[h] \in (\Z/k_1k_2\Z)^{\times}$
\begin{multline}\label{Wanted-congruence}
2s(k_2h, k_1) - s(2k_2 h, k_1) - \frac{2n_1 k_2 h}{k_1} - 2s(k_1 h, k_2) + s(2k_1 h, k_2) + \frac{2n_2 k_1 h}{k_2} \\ 
- 2s(h, k_1 k_2) + s(2h, k_1 k_2) + \frac{2n h}{k_1 k_2} \equiv 0 \pmod{2}
\end{multline}

Now we want to apply Theorem \ref{Rademacher-Whiteman}. In order to do this we
let $w \in \Z$ be given by 
$$
w := \frac{24}{\gcd{(24, k_1 k_2)}},
$$ 
Since $\gcd{(w, k_1k_2)} = 1$,
we can assume without loss of generality that in \eqref{Wanted-congruence} we have that $w \mid h$.
Under this assumption, since $\gcd{(k_1, k_2)} = 1$ we have that $24 \mid k_1 k_2 h$.
Thus we can apply Theorem \ref{Rademacher-Whiteman} twice 
and we get
\begin{align}\label{Conguence-Rademacher-Whiteman-1}
\left( s(k_2h, k_1) - \frac{k_2 h}{12k_1} \right) + \left( s(k_1h, k_2) - \frac{k_1 h}{12k_2} \right) - \left( s(h, k_1k_2) - \frac{h}{12k_1 k_2} \right) & \equiv 0 \pmod{2}\\
\label{Conguence-Rademacher-Whiteman-2}
\left( s(2k_2h, k_1) - \frac{2k_2 h}{12k_1} \right) + \left( s(2k_1h, k_2) - \frac{2k_1 h}{12k_2} \right) - \left( s(2h, k_1k_2) - \frac{2h}{12k_1 k_2} \right) & \equiv 0 \pmod{2}.
\end{align}
Adding together the congruences (\ref{Wanted-congruence}) and (\ref{Conguence-Rademacher-Whiteman-2}) and subtracting twice the congruence (\ref{Conguence-Rademacher-Whiteman-1}) from that, we get that \eqref{Wanted-congruence} is equivalent to
\begin{align*}
    2\left(\frac{nh}{k_1k_2}-\frac{n_1hk_2}{k_1}-\frac{n_2hk_1}{k_2}\right) \equiv 0 \pmod{2},
\end{align*}
which, since $\gcd(h,k_1 k_1) = 1$, is in turn equivalent to
\begin{align*}
n - n_1k_2^2 - n_2k_1^2 \equiv 0 \pmod{k_1k_2}.
\end{align*}
By the Chinese Remainder Theorem, this equivalent to the system of congruences
\begin{align*}
n &\equiv n_2k_1^2 \pmod{k_2},\\
n&\equiv n_1k_2^2 \pmod{k_1},
\end{align*}
which completes the proof.
\end{proof}


\subsection{Proof of Theorem \ref{P-Power}}

The main idea of the proof will be to express $\widetilde{A}_{k}(n)$ as a Salié sum, and then, to apply a theorem of Salié \cite{Sal32} that evaluates such sums explicitly. 

\medskip

We start by introducing some notation.
For $h, k \in \Z$ with $k \geq 1$ odd, $0 \leq h < k$ and $\gcd(h, k) = 1$ we define
\begin{align}\label{gn}
g_n(h, k) := 24k\,s(h, k) - 12k\,s(2h, k) -24nh,
\end{align}
so that \eqref{Aktilde_bis} can be written in terms of $g_n$ as follows:
\begin{align}\label{Ak-gn}
\A_k(n) = \sum_{\substack{0 \leq h < k\\ (h, k) = 1}} \exp{\left( \frac{\pi i}{12k} g_n(h, k) \right)}.
\end{align}

Note that by part (\ref{item:s-entero}) of Proposition \ref{DedekindSumProperties} we have that $g_n(h, k)$ is 
an even integer. 
Moreover, in the next lemmas we will use the congruences satisfied by the Dedekind sums $s(h, k)$ listed in Proposition \ref{DedekindSumProperties} in order to prove certain congruences satisfied by $g_n(h, k)$.

\begin{lemma}
The integer $g_n(h, k)$ satisfies that
\begin{equation}
\label{eqn:gn_mod8}
g_n(h, k) \equiv 
-6k\legendre{h}{k} + 3k\left(1+\legendre{-1}k \right)
\pmod{8}.
\end{equation}
\end{lemma}

\begin{proof}
    Using the congruence from Proposition \ref{DedekindSumProperties} (\ref{item:s-mod8}) in the definition of $g_n(h, k)$ we have
\begin{align*}
g_n(h, k) &\equiv 24k\,s(h, k) - 12k\,s(2h, k) -24nh \pmod{8}\\
&\equiv 2 \left( k + 1 - 2\legendre{h}{k} \right) - \left( k + 1 - 2\legendre{2h}{k} \right) \pmod{8}\\
&\equiv 
k+1 - \legendre{h}{k}\left(4-2\legendre{2}{k}\right) \pmod{8}.
\end{align*}
Thus, it suffices to prove that for $\delta \in \{\pm1\}$ we have that
\begin{equation*}
    -6k\delta + 3k\left(1+\legendre{-1}k \right)
    \equiv
k+1 - \delta\left(4-2\legendre{2}{k}\right) \pmod{8}.
\end{equation*}
The latter equation can be easily verified for every odd $k$ modulo $8$, hence the result follows.
\end{proof}

For the remainder of this section we fix $n$ and $k$ as above, and for each integer $h$ with $0 \leq h < k$ with $\gcd(h,k) = 1$ we denote
$$t_n(h) = -nh + \frac1{16h} \pmod{k}.
$$

\begin{lemma}
The integer $g_n(h, k)$ satisfies that
\begin{equation}
\label{eqn:gn_mod3k}
g_n(h, k) \equiv 24 t_n(h) \pmod{3k}.
\end{equation}
\end{lemma}

\begin{proof}
    Let $\theta := \gcd(k,3)$.
    By Proposition \ref{DedekindSumProperties} (\ref{item:s-h2}) we have that 
\begin{align}\label{JeanCarlos-Congruence}
12k\,s(h, k) \equiv h + h^{-1} \pmod{\theta k}.
\end{align}
Using (\ref{JeanCarlos-Congruence}) both for $h$ and $2h$ we get that
\begin{align*}
g_n(h, k) &= 24k\,s(h, k) - 12k\,s(2h, k) -24nh\\
&\equiv 2(h + h^{-1}) - (2h + (2h)^{-1}) - 24nh \pmod{\theta k}\\
&\equiv h^{-1}(2 - 2^{-1}) - 24nh \pmod{\theta k}\\
&\equiv 3(2h)^{-1} - 24nh \pmod{\theta k}\\
&\equiv 24t_n(h) \pmod{\theta k}.
\end{align*}
This proves the result in the case $\theta = 3$.

When $\theta = 1$, the above congruence shows that $g_n(h,k) \equiv 24 t_n(h) \pmod k$.
Then it suffices to prove that $g_n(h,k) \equiv 0 \pmod 3$.
By definition we have that
\begin{align*}
g_n(h, k) &\equiv 24k\,s(h, k) - 12k\,s(2h, k) \pmod{3}.
\end{align*}
Since $\theta = 1$, by Proposition \ref{DedekindSumProperties} (\ref{item:s-mod3}) we have that
$12k\,s(h, k) \equiv 12k\,s(2h,k) \equiv 0 \pmod{3}$.
Hence the result follows.
\end{proof}





Let $k$ be, as above, a positive, odd integer, and let $a$ be an integer modulo $k$.
The corresponding \emph{Salié sum} is defined by
\begin{align}\label{Salie-sum-definition}
S(a, k) := \sum_{\substack{0 \leq h < k\\ (h, k) = 1}} \dlegendre{h}{k} e^{2 \pi i (ah + h^{-1})/k}.   
\end{align}
In the following proposition we show that 
$\widetilde{A}_k(n)$ can be written as a Salié sum.

\begin{proposition}
\label{prop:AtSalie}
Let $\epsk$ be defined by
\begin{equation*}
    \epsilon_k := 
    (-i)^{((k-1)/2)^2}
    = 
\begin{cases}
1, &\text{if $k \equiv 1 \pmod{4}$},\\
-i, &\text{if $k \equiv -1 \pmod{4}$}.
\end{cases}
\end{equation*}
Let $16^{-1}$ be an inverse of $16$ modulo $k$, and let  $a := -16^{-1}n$.
Then
$$
\widetilde{A}_k(n) = \epsk
S(a, k).
$$
\end{proposition}

\begin{proof}
    Let $0 \leq h < k$ with $\gcd(h, k) = 1$.
    The congruences in \eqref{eqn:gn_mod8} and \eqref{eqn:gn_mod3k} imply that
    \begin{equation*}
        g_n(h,k) \equiv
        24 t_n(h)
        - 6k\legendre{h}{k}
        + 3k\left(1+\legendre{-1}{k}\right)
        \pmod{24k}.
    \end{equation*}
    From this we get that
    \begin{align*}
    \exp{\left(\frac{\pi i}{12k}g_{n}(h,k)\right)}
    &=
    \exp{\left(\frac{\pi i(1+\legendre{-1}{k})}{4}\right)}
    \exp{\left(-\frac{\pi i\legendre{h}{k}}{2}\right)}
    \exp{\left(\frac{2\pi it_n(h)}{k}\right)}
\end{align*}
Then using that
\begin{align*}
    \epsk = -i \exp{\left(\frac{\pi i(1+\legendre{-1}{k})}{4}\right)}
    \qquad \text{and} \qquad
    \legendre{h}{k} = i\exp{\left(-\frac{\pi i\legendre{h}{k}}{2}\right)}
\end{align*}
we get from \eqref{Ak-gn} that
$$
\widetilde{A}_k(n) =
\epsk
\sum_{\substack{0 \leq h < k\\ (h, k) = 1}} \dlegendre{h}{k} e^{2\pi i t_n(h) / k},
$$
Hence the result follows by considering the change of variables $s = 16 h$, so that $t_n(h) \equiv as + s^{-1} \pmod k$ and $\legendre{s}{k} = \legendre{h}{k}$.
\end{proof}

We now state the theorem of Salié that we will use for computing Salié sums in the case of a prime power.
We note that here we cite the version of Salié's results given by Lehmer \cite[Lemma* 2, p. 283]{Leh38}, but, as Lehmer indicates, the results of that lemma can be gathered from equations (32), (37), (54) and (57) of \cite{Sal32}.

\begin{theorem}[Salié]\label{Salie-sum-evaluation}

Let $a$ be an integer, and let $q$ be a power of an odd prime.
Then we have
\begin{enumerate}
\item $S(a, q) = \epsq^{-1} \sqrt{q}$,\quad if $p | a$ and $\alpha = 1$.
\item $S(a, q) = 0$,\quad if $p | a$ and $\alpha > 1$.
\item $S(a, q) = 0$,\quad if $\gcd(a, q) = 1$ and $a$ is not a quadratic residue modulo $q$.
\item $S(a, q) = 2 \epsq^{-1} \sqrt{q} \cos{\left( \dfrac{4 \pi \theta}{q} \right)}$,\quad if $\gcd(a, q) = 1$ and $\theta \in \Z$ satisfies $\theta^2 \equiv a \pmod{q}$.

\end{enumerate}

\begin{proof}[Proof of Theorem~\ref{P-Power}]
    The theorem follows immediately from Proposition~\ref{prop:AtSalie} and Theorem~\ref{Salie-sum-evaluation}.
\end{proof}

\end{theorem}

\section{Algorithms}
\label{sect:algorithms}

\SetKwComment{Comment}{/* }{ */}
\SetKw{KwReturn}{return}
\SetKwInput{KwData}{Input}
\SetKwInput{KwResult}{Output}

We first recall from formula \eqref{R-Definition1} and Theorem \ref{thm:explicit-error-bound} that $\ov(n)$ is given by
\begin{align}
\label{Engel2}
\ov(n) = \frac{1}{4n} \sum_{\substack{k = 1\\ 2 \nmid k}}^N \frac{1}{\sqrt{k}} \,\A_k(n) \,U\left( \frac{\ct(n)}{k} \right) + \Rt(n, N),
\end{align}
where the error term $\Rt(n, N)$ satisfies the bound
\begin{align}
\label{eqn:error_bound_bis}
    \Rt(n,N) \leq \frac{1}{4\pi} \left( \frac{N+1}{n} \right)^{3/2} \left( \frac{\pi \sqrt{n}}{N+1} \cosh{ \left( \frac{\pi \sqrt{n}}{N+1} \right) } + (2N + 1) \sinh{ \left( \frac{\pi \sqrt{n}}{N+1} \right) } - 2\pi \sqrt{n} \right).
\end{align}

First, in Algorithm \ref{alg:mult}, we describe a method for computing the values $\A_k(n)$, based on the
multiplicativity result from Theorem~\ref{Multiplicativity} and the formulas for
the case when $k$ is a prime power given in Theorem~\ref{P-Power}.


\begin{algorithm}[H]
\caption{Evaluation of $\A_{k}(n)$}
\label{alg:mult}

\KwData{Integers $k\geq 1$, $n \geq 1$, with $k$ odd}

\KwResult{$\A_{k}(n)$, as defined in (\ref{Aktilde})}

\BlankLine


$k \gets p_{1}^{\alpha_{1}}\cdots p_{j}^{\alpha_{j}}$



$n_3 \gets n$

$k_2 \gets k$

$s \gets 1$

$i \gets 1$

\While{$s\neq 0$ and $k_2 \neq 1$}{
    $k_1 \gets p_{i}^{\alpha_{i}}$
    
    $k_2 \gets k_2 / k_1$ 
    
    $n_1 \gets n_3 / k_2^2 \pmod{k_1}$
    
    $n_2 \gets n_3 / k_1^2 \pmod{k_2}$
    
    
    $s \gets s \cdot \A_{k_1}(n_1)$ \Comment*[r]{Use Theorem~\ref{P-Power}}
    \label{step:primepow}
    $n_3 \gets n_2$
    
    $i \gets i + 1$
}
\KwReturn{$s$}
\end{algorithm}

\begin{remark}
	\label{rmk:return0}
Algorithm~\ref{alg:mult} will terminate, and return $0$, whenever the value $\A_{k_1}(n_1)$
computed in Step~\ref{step:primepow} is $0$: namely when $p_i \mid n_1$ and
$\alpha_i > 1$ or when $p_i \nmid n_1$ and $-n_1$ is not a square modulo $k_1$.
Thus, at least when $\gcd(k,n) = 1$, we expect the $k$-th summand of \eqref{Engel2} to be non-zero only $1/2^j$
of the times, where $j$ denotes the number of prime divisors of $k$.
\end{remark}
In order to obtain $\ov(n)$ from \eqref{Engel2} we follow \cite{Joh12}, where
the author uses similar formulas for computing the partition function.

We start by letting $N = \lceil \sqrt{n} \rceil$.
The bound in \eqref{eqn:error_bound_bis} implies that $\Rt(n,N) < \frac14$ for $n > 784$.
We will assume that this is the case from now on; for the remaining values of
$n$ we can easily compute $\ov(n)$ by expanding \eqref{eqn:etaq} 
or using the recursive formulas from Section~\ref{sect:recursive}.

We denote by $t_k$ the $k$-th summand in \eqref{Engel2}, 
and denote by $\htk$ a floating point approximation of $t_k$.
If the working precision is enough so that
\begin{equation}
    \label{eqn:tk_error}
    \abs{t_k - \htk} < \frac1{4N}
\end{equation}
then we have that
\begin{equation}
    \label{eqn:pbar_sumtk}
    \abs{\ov(n) - \sum_{k=1}^N \htk} < \frac12,
\end{equation}
which allows to determine the value of the integer $\ov(n)$.

Algorithm~\ref{alg:mult} shows that either $t_k = 0$ (see
Remark~\ref{rmk:return0}) or there exists a subset $\mathcal{P}_k$ of size $m_k$
of the primes dividing $k$ such that 
\begin{equation}
    \label{eqn:tk}
    t_k =
        \frac1{4n} \,
        U\left(\frac{\ct(n)}k\right) \,
        \prod_{p \in \mathcal{P}_k} 
            2 \cos\left(\frac{4\theta_p}{p^{\alpha_p}}\right)
\end{equation}
Note that, in particular, the square root in \eqref{Engel2} gets cancelled.
Following the lines of the proof of \cite[Theorem 4]{Joh12} (which requires $n > 2000$) we get the following
result.
\begin{proposition}
	\label{prop:precision}
	For \eqref{eqn:tk_error} to hold, it is sufficient to evaluate
	\eqref{eqn:tk} with a precision of
	$$
    r_k = \max
    \left\{
    -\tfrac12 \log_2(n) + \pi \tfrac{\sqrt{n}}k \log_2(e) + m_k
    + \log_2\left(10\pi\tfrac{\sqrt{n}}k + 7(2m_k-4)\right),
    \tfrac12\log_2(n) + 5,
    11
    \right\}
    $$
    bits.
\end{proposition}

We summarize this discussion with the following algorithm.


\begin{algorithm}
\caption{Evaluation of $\ov(n)$}\label{alg:overpartitions}

\KwData{Integer $n > 2000$\Comment*[r]{Otherwise use the recursive formula}} 

\KwResult{$\ov(n)$}

\BlankLine

$N \gets \lceil{\sqrt{n}}\rceil$

$\ov \gets 0$

\For{$k\gets 1$ \KwTo $N$}{\If{$t_k \neq 0$}{
    $\htk \gets$ compute $t_k$ with precision $r_k$
    
    $\ov \gets \ov + \htk$
}
}

$\ov \gets$ round $\ov$ to the nearest integer

\KwReturn{$\ov$}

\end{algorithm}

\begin{remark}
	Since we are using \eqref{eqn:pbar_sumtk} to compute $\ov(n)$, the working
	precision in Algorithm~\ref{alg:overpartitions} must be at least the size of
	$\ov(n)$. 
	The latter can be bounded by \eqref{Order-Calculation}, which implies that for
	$n$ as above we have
	\[
        \abs{\ov(n)} \leq \frac{e^{\pi\sqrt{n}}}{8n}.
    \]
\end{remark}

We conclude this section with a slight modification of
Algorithm~\ref{alg:overpartitions} for computing $\ov(n)$ modulo an odd integer
$m$.
This is enough for proving congruences for the overpartition function, and is
less memory intensive.

The key observation is the following: write
\[
	\ov(n) = a + s + \varepsilon
\]
where $a \in \Z, s \in [0,1)$ and $\varepsilon < 1$.
Then
\[
	s+\varepsilon = \ov(n) - a \in (-1,2) \cap \Z = \{0,1\}.
\]

Since $\ov(n) \equiv 0 \pmod 2$, knowing $a \pmod {2m}$ determines first the
value of $s+\varepsilon$ and then, since $m$ is odd, the value of $\ov(n) \pmod
m$.


\begin{algorithm}
	\caption{Evaluation of $\ov(n) \pmod m$}
	\label{alg:overpartitionsmodm} 

\KwData{Integers $n > 2000, m \geq 3$, with $m$ odd}

\KwResult{$\ov(n) \pmod m$}

\BlankLine

$N \gets \lceil{\sqrt{n}}\rceil$

$\ov \gets 0 \pmod{2m}$

$b \gets 0$

\For{$k\gets 1$ \KwTo $N$}{\If{$t_k \neq 0$}{
    $\htk \gets$ compute $t_k$ with precision $r_k$

	$a_k \gets \lfloor \htk \rfloor$

	$b_k \gets \htk - a_k$

	$b \gets b_k + b$

	\If{$b > 1$}{
	$b \gets b-1$

	$a_k \gets a_k + 1$

}
    
	$\ov \gets \ov + a_k \pmod{2m}$

}
}

\If{
	$\ov \equiv 1 \pmod 2$
}
{
	$\ov \gets \ov + 1$
}

$\ov \gets \ov \pmod m$

\KwReturn{$\ov$}

\end{algorithm}

\begin{remark}
	In this algorithm the working precision must be set so that
	the sum of the rounding errors and $\Rt(n,N)$ have absolute value less than
	$1$.
             By \eqref{eqn:error_bound_bis} we get that $\abs{\Rt(n,N)} < \tfrac12$ for $n>36$, hence for such $n$ it suffices to work with precision such that $\abs{t_k - \htk} < \tfrac1{2N}$.

\end{remark}

\section{New congruences}

In \cite[Theorem 4.2]{RSST21} the third and fifth authors, along with two other co-authors, described a method for finding congruences
for arbitrary eta-quotients, which we now state in its simplified version for
the particular case of the overpartition function.  See Algorithm~\ref{alg:interesting}.

We let $\ell > 2$ be a prime. Let us denote
\begin{equation}\label{eqn:kl}
	k_\ell = \begin{cases}
		24, & \ell = 3, \\
		\ell^2-1, & \ell \geq 5.
		\end{cases}
\end{equation}
Furthermore, given a non-zero integer $c$ with $c \mid 16 \ell^2$ and $\ell^2
\nmid c$ we denote
\begin{equation}\label{eqn:vanishing_Fp}
	f_{c,\ell} =
	\frac{16}{\gcd(c^2,16)} \cdot
	\begin{cases}
		10   & \text{if } \ell = 3 \text{ and } \ell \nmid c, \\
		1   & \text{if }  \ell = 3 \text{ and } \ell \mid c, \\
		\tfrac{\ell^4-1}{24} & \text{if } \ell \geq 5 \text{ and } \ell\nmid c, \\
		\tfrac{\ell^2-1}{24} & \text{if } \ell \geq 5 \text{ and } \ell\mid c.
	\end{cases}
\end{equation}
For non-integral values of $m$ we let $\ov(m) = 0$.
Finally, given an integer $j \geq 1$ we say that a prime $Q$ is a
\emph{candidate} (for yielding congruences) if $Q \equiv -1 \pmod{16 \ell^j}$.
We remark that the repeated occurrence of 16 in the above is because the eta
quotient whose coefficients are $\ov(n)$ is a weakly holomorphic modular form of level 16.

\begin{algorithm}

\LinesNumbered

\KwData{An odd prime $\ell$;
	 an integer $j\geq 1$;
	a candidate $Q$.}
\KwResult{\texttt{True} or \texttt{False}}

\BlankLine

$\beta \leftarrow j-1$

\For{every cusp $s = a/c$ for $\Gamma_0(16\ell^2)$ with $c \mid 16\ell^2$ and
$\ell^2 \nmid c$}{

   $\beta \leftarrow \max\{\beta,\left \lceil -\log_\ell\left(f_{c,\ell}\right)\right \rceil\}$


}

$\kappa\leftarrow -1 + \ell^\beta k_\ell$

$\delta \leftarrow (-1)^{\tfrac{\kappa-1}{2}}$

$n_0 \leftarrow \frac{\kappa (\ell + 1)}2 + 1 \label{step:sturm}$ 

interesting $\leftarrow$ \texttt{True}

\While{interesting}{
	\For{$n\leftarrow 1$ \KwTo $n_0$}{\If{$\legendre{-n}\ell = -1$}{
			\If{$\ov\left(n Q^2\right) +
				\legendre{\delta n}{Q} Q^{\tfrac{\kappa-3}{2}} \, \ov(n) +
			Q^{\kappa-2} \, \ov\left(n/{Q^2}\right)\not\equiv 0
		\pmod {\ell^j}$\label{step:cong}}
{interesting$\leftarrow$ \texttt{False}}
}
}
}
\Return interesting
\BlankLine

\caption{Algorithm for finding congruences}
\label{alg:interesting}
\end{algorithm}

\begin{proposition}
	\label{thm:interesting}
 
	If the output of Algorithm \ref{alg:interesting} is \texttt{True},
	then
	\begin{equation*}
		\ov\left(Q^3 n\right) \equiv 0 \pmod{\ell^j}
	\end{equation*}
	for all $n \geq 1$ coprime to $\ell Q$ such that 
	$\legendre{n}{\ell} = -1$.
\end{proposition}

\begin{remark}
    In \cite[Proposition 1.5]{Treneer08} it is shown that, given a pair $(\ell,j)$, for a positive proportion of the candidates $Q$ the congruences in Proposition~\ref{thm:interesting} should hold.
\end{remark}

The key part of Algorithm~\ref{alg:interesting} is to verify the congruences in
Step \ref{step:cong}.
In \cite[Proposition 5.2]{RSST21} we were able to do this only by the naive
method, namely by expanding, modulo $\ell^j$, the infinite product
\eqref{eqn:etaq}.
In particular, this implied that we were computing $\ov(n)$ for $1 \leq n \leq
Q^2 n_0$, where $n_0$ is the Sturm bound given by Step \ref{step:sturm}, rather
than computing only the values that we actually needed, namely,
\begin{equation}
	\label{eqn:3valores}
	\ov\left(n Q^2\right),\;
	\ov(n),\;
	\text{and }
	\ov\left(n/Q^2\right),
	\qquad 
	\text{ for }
	1 \leq n \leq n_0,
	\;
	\legendre{-n}\ell = -1.
\end{equation}
This cuts of the number values from $Q^2 n_0$ to, roughly, half of  $3 n_0$.
Moreover, for small Sturm bounds (which is the case of the examples we are able
to treat here), of the three values in \eqref{eqn:3valores} only
$\ov\left(nQ^2\right)$ is non-trivial computationally.

Here we take advantage of the possibility of computing efficiently individual
values of the overpartition function modulo $\ell^j$ given in
Algorithm~\ref{alg:overpartitionsmodm} to extend the results given in
\cite[Proposition 5.2]{RSST21}.
The congruences we obtained are the following.

\begin{theorem}\label{thm:congruences1}
    Let $(\ell,j) \in \{(3,1),(3,2),(5,1),(5,2)\}$.
	We have that
	\[
		\overline{p}\left(Q^3n\right)\equiv 0 \pmod{\ell^j}
	\]
	for all $n$ coprime to $\ell Q$ such that $\legendre{n}{\ell} = -1$, for every candidate $Q < 10^5$.
 \end{theorem}

\begin{remark}
We claimed in \cite{RSST21} that for $(\ell,j)$ as above there were candidates
for which Algorithm~\ref{alg:interesting} returned \texttt{False} (e.g., for
$\ell = 3, j = 1, Q = 1151$); this was due to a mistake in the implementation of
Algorithm~\ref{alg:interesting} used in \cite{RSST21}, which is now corrected in \cite{codeRSST}.
\end{remark}
 
\begin{theorem}\label{thm:congruences2}
	We have that
	\[
		\overline{p}\left(Q^3n\right)\equiv 0 \pmod{\ell^j}
	\]
	for all $n$ coprime to $\ell Q$ such that $\legendre{n}{\ell} = -1$, for 
	\begin{itemize}
		\item $\ell = 3, j = 3$ and
			\[
				Q = 2591, 4751.
			\]
		\item $\ell = 7, j = 1$ and
			\[
				Q = 1231, 2239, 3023, 4703, 5039, 9743.
			\]
	\end{itemize}

In each case, these are all the candidates $Q < 10^5$ for which Algorithm~\ref{alg:interesting} outputs \texttt{True}.
\end{theorem}

We made our computations using \cite{sagemath}.
The code is available at \cite[v1.0]{code}, where we also store the values in \eqref{eqn:3valores} for each case where Algorithm \ref{alg:interesting} outputs \texttt{True}, as well as witnesses in case it outputs \texttt{False}.

\subsection{Further questions}

\begin{enumerate}

    \item
    For $(\ell,j) \in \{(3,1),(3,2),(5,1),(5,2)\}$, do there exist candidates $Q$ for which Algorithm~\ref{alg:interesting} outputs \texttt{False}?
    We could not find such $Q$ for $Q < 10^5$, among the 75 candidates for $(3,1)$, the 23 candidates for $(3,2)$, the 36 candidates for $(5,1)$ and the 4 candidates for $(5,2)$.
    \item
    For $\ell > 7$, do there exist candidates $Q$ for which Algorithm~\ref{alg:interesting} outputs \texttt{True}?
    Among the 114 candidates for $\ell = 11$ satisfying $Q < 10^6$, we found none that returned \texttt{True}.
\end{enumerate}

\end{document}